\documentclass[11pt]{article}

\usepackage{color}
\usepackage[dvips]{graphicx}
\usepackage{amsmath}
\usepackage{amsthm}
\usepackage{dsfont}

\theoremstyle{plain}
\newtheorem{theorem}{Theorem}

\newtheorem{lemma}{Lemma}

\newtheorem{proposition}[lemma]{Proposition}

\theoremstyle{definition}

\newtheorem{definition}[lemma]{Definition}

\theoremstyle{remark}

\newtheorem{remark}[lemma]{Remark}

\begin{document}

\title{Piecewise analytic bodies in subsonic potential flow}%

\author{Volker Elling}%
\date{}%
\maketitle%

\begin{abstract}
  We prove that there are no nonzero uniformly subsonic potential flows around bodies with three or more protruding corners,
  for piecewise analytic boundary and for equation of state a $\gamma$-law with $\gamma>1$. 
  This generalizes an earlier result limited to the low-Mach limit for nondegenerate polygons. 
  For incompressible flows we show the velocity cannot be globally bounded. 
\end{abstract}

\newcommand{\isect}{\cap}
\newcommand{\upconv}{\nearrow}
\newcommand{\dnconv}{\searrow}
\newcommand{\topref}[2]{\overset{\text{\eqref{#1}}}{#2}}
\newcommand{\csep}{\quad,\quad}
\newcommand{\eps}{\epsilon}
\newcommand{\pd}[1]{\partial_{#1}}
\newcommand{\closure}[1]{\overline{#1}}
\newcommand{\spC}{\mathcal{C}}
\newcommand{\Ck}[1]{\spC^{#1}}
\newcommand{\Ctwo}{\Ck2}
\newcommand{\Cinf}{\Ck\infty}
\newcommand{\conv}{\rightarrow}
\newcommand{\esssup}{\operatorname{esssup}}
\newcommand{\Leb}{\mathcal L}
\newcommand{\Leba}[1]{\Leb^{#1}}
\newcommand{\Linf}{\Leba\infty}
\newcommand{\set}[1]{\{#1\}}
\newcommand{\selset}[2]{\{#1~:~#2\}}
\newcommand{\setdiff}{\backslash}
\newcommand{\bdry}{\partial}
\newcommand{\R}{\mathds{R}}
\newcommand{\const}{\text{const}}
\newcommand{\vlen}[1]{|#1|}
\newcommand{\trace}{\operatorname{tr}}
\newcommand{\pt}{\partial_t}
\newcommand{\ndiv}{\nabla\dotp}
\newcommand{\dotp}{\cdot}
\newcommand{\crossp}{\times}
\newcommand{\ncurl}{\nabla\crossp}
\newcommand{\hess}{\nabla^2}
\newcommand{\half}{\frac12}

\newcommand{\tensor}{\otimes}

\newcommand{\defm}[1]{\emph{#1}}
\newcommand{\pglau}{\beta} 
\newcommand{\polaeps}{\eps}
\newcommand{\gC}{C}
\newcommand{\Machbound}{\Mach_{\min}}
\newcommand{\ent}{s}
\newcommand{\epm}{e}
\newcommand{\qpm}{q}
\newcommand{\tmpr}{T}

\newcommand{\extRtwo}{\hat\R^2}

\newcommand{\Caa}{C_1}%
\newcommand{\Cab}{C_2}%
\newcommand{\Cac}{C_3}%
\newcommand{\Cad}{C_4}%
\newcommand{\Cae}{C_5}%

\newcommand{\vv}{\vec v}
\newcommand{\vscal}{|\vv|}
\newcommand{\Uci}{S}
\newcommand{\rhsf}{F}
\newcommand{\sstf}{\overline\stf}
\newcommand{\scl}{\eps}
\newcommand{\nconv}{\conv}
\newcommand{\ssndz}{\ssnd_{\max}}
\newcommand{\densz}{\dens_{\max}}
\newcommand{\sstflim}{\sstf}
\newcommand{\xxlim}{\xx\scriptinf}

\newcommand{\chaball}[2]{\overline H_{#1}(#2)}%
\newcommand{\haball}[2]{H_{#1}(#2)}%
\newcommand{\arads}{\varrho}%
\newcommand{\arad}[1]{\arads(#1)}%
\newcommand{\haballc}{\haball{\arad0}{0}}%
\newcommand{\xxd}{\xx_1}%
\newcommand{\aad}{\aa_1}%
\newcommand{\haballd}{\haball{\arad{\aad}}{\aad}}%
\newcommand{\aae}{\aa_2}%
\newcommand{\aaf}{\aa_*}%
\newcommand{\aradf}{|\aa-\aaf|}%
\newcommand{\haballf}{\haball{\aradf}{\aaf}}%

\newcommand{\rmin}{\underline\rad}
\newcommand{\Pola}{\Theta}
\newcommand{\xrad}{\rho}
\newcommand{\morex}{\delta}
\newcommand{\Ig}{J}
\newcommand{\CPW}{C_{PW}}
\newcommand{\plam}{\pd\lam}
\newcommand{\pLam}{\pd\Lam}
\newcommand{\Jlrv}{B^{(\lrv)}}
\newcommand{\Jpola}{B^{(\pola)}}
\newcommand{\maxf}{M_\ff}
\newcommand{\Dir}{\mathcal{D}}

\newcommand{\me}{m_\epsilon}
\newcommand{\Ue}{U_\epsilon}
\newcommand{\Ke}{K_\epsilon}
\newcommand{\We}{W_\epsilon}

\newcommand{\quaf}{\quasic_\ff}
\newcommand{\ff}{\vec f}
\newcommand{\fg}{\vec F}
\newcommand{\fx}{f^x}
\newcommand{\fy}{f^y}
\newcommand{\fxavg}{\overline{\fx}}
\newcommand{\fyavg}{\overline{\fy}}
\newcommand{\ffavg}{\overline{\ff}}
\newcommand{\ffvar}{\tilde\ff}
\newcommand{\fxvar}{\tilde\fx}
\newcommand{\fyvar}{\tilde\fy}

\newcommand{\mineig}{\lambda}
\newcommand{\maxeig}{\Lambda}
\newcommand{\ssig}{\sigma}
\newcommand{\Ber}{B}
\newcommand{\isenc}{\gamma} 
\newcommand{\Ctb}{\Ck3}
\newcommand{\lamlo}{0}
\newcommand{\lamhi}{1}
\newcommand{\polalo}{\pola_0}
\newcommand{\polahi}{\pola_1}
\newcommand{\polaa}{\pola_\lam}
\newcommand{\lrvlo}{\lrv_0}
\newcommand{\lrvhi}{\lrv_1}
\newcommand{\lrvj}{\lrv_j}
\newcommand{\polaj}{\pola_j}
\newcommand{\intp}{\int_{\lamlo}^{\lamhi}}
\newcommand{\intl}{\int_{\lrvlo}^{\lrvhi}}
\newcommand{\aacorner}{\aa_c}
\newcommand{\bighhball}{\cball{R/2}{\aacorner}\isect\haplane}
\newcommand{\lam}{q}
\newcommand{\Lam}{\vec\lam}
\newcommand{\quasic}{K}
\newcommand{\elli}{\kappa}
\newcommand{\Corners}{\set{\text{corners}}}
\newcommand{\Cutoff}{\Bdry_{\rmin}}
\newcommand{\haplane}{\R^2_+}
\newcommand{\habdry}{\bdry\R^2_+}
\newcommand{\ww}{\ff}
\newcommand{\www}{w}
\newcommand{\wwi}{w_\infty}
\newcommand{\wpot}{\Phi}
\newcommand{\prho}{\pd\rho}
\newcommand{\pal}{\pd\varphi}
\newcommand{\subso}{\underline\stf}
\renewcommand{\aa}{\vec a}
\newcommand{\av}{a}
\newcommand{\bv}{b}
\newcommand{\gv}{\vec g}
\newcommand{\varstf}{\tilde\stf}
\newcommand{\vxi}{v^x_\infty}
\newcommand{\stfnew}{\stf^\text{new}}
\newcommand{\varstfnew}{\varstf^\text{new}}
\newcommand{\varvxnew}{\tilde v^{x,\text{new}}}
\newcommand{\varvynew}{\tilde v^{y,\text{new}}}
\newcommand{\vxnew}{v^{x,\text{new}}}
\newcommand{\vynew}{v^{y,\text{new}}}
\newcommand{\varvx}{\tilde v^x}
\newcommand{\varvy}{\tilde v^y}
\newcommand{\Ps}{\set{\stf>0}}
\newcommand{\Zs}{\Dom_0}
\newcommand{\Zl}{Z_-}
\newcommand{\Zr}{Z_+}
\newcommand{\Ns}{\set{\stf<0}}
\newcommand{\cDom}{\closure\Dom}
\newcommand{\cDomi}{\closure\Dom\union\set\infty}
\newcommand{\sDom}{\Dom\union\Slipb}
\newcommand{\sDomi}{\Dom\union\Slipb\union\set\infty}
\newcommand{\Dom}{\Omega}
\newcommand{\Bdry}{\bdry\Dom}
\newcommand{\Slipb}{\Bdry}
\newcommand{\vx}{v^x}
\newcommand{\vy}{v^y}
\newcommand{\lrv}{a}
\newcommand{\Lrv}{\mathbf a}
\newcommand{\cz}{\overline z}
\newcommand{\pcz}{\pd\cz}
\newcommand{\pz}{\pd z}
\newcommand{\Hf}{H}
\newcommand{\Gam}{\Gamma}
\newcommand{\Dt}{D_t}
\newcommand{\vm}{\vec m}
\newcommand{\vmi}{\overline\vm}
\newcommand{\vmq}{|\vec m|^2}
\newcommand{\vms}{\mu}
\newcommand{\hdiv}{\hat\tau}
\newcommand{\vmssonic}{\vms_1}
\newcommand{\Body}{B}
\newcommand{\vmsmax}{\overline\vms}
\newcommand{\gdiv}{\vec a}
\newcommand{\gmod}{\vec{\tilde a}}
\newcommand{\hmod}{\tilde\tau}
\newcommand{\dhmod}{\tilde\tau'}
\newcommand{\hexp}{\zeta}
\newcommand{\dhdiv}{\hat\tau'}
\newcommand{\vmsonic}{m_1}
\newcommand{\vcav}{v_*}
\newcommand{\rhomax}{\rho_0}
\newcommand{\vort}{\omega}
\renewcommand{\Re}{\operatorname{Re}}
\renewcommand{\Im}{\operatorname{Im}}
\newcommand{\zacs}[1]{#1_0}
\newcommand{\zac}[1]{(#1)_0}
\newcommand{\lacs}[1]{\delta#1}
\newcommand{\lac}[1]{\delta(#1)}
\newcommand{\dens}{\varrho}
\newcommand{\vpot}{\phi}
\newcommand{\cpot}{\Phi}
\newcommand{\stf}{\psi}

\newcommand{\piv}{p}
\newcommand{\pif}{\hat\piv}
\newcommand{\ipif}{\pif^{-1}}
\newcommand{\pp}{P}
\newcommand{\ppf}{\hat\pp}

\newcommand{\Lip}{\operatorname{Lip}}
\newcommand{\Cs}{\Ck\sobex}
\newcommand{\sobex}{s}

\newcommand{\xiv}{\lambda}
\newcommand{\leba}[1]{L^{#1}}
\newcommand{\sobs}{\soba\sobex}
\newcommand{\soba}[1]{H^{#1}}
\newcommand{\sobexi}{m}
\newcommand{\sobexf}{\sigma}
\renewcommand{\Lip}{\Ck{0,1}}
\newcommand{\Cm}{\Ck{\sobexi}}
\renewcommand{\Cs}{\Ck{\sobexi,\sobexf}}
\newcommand{\iCk}[1]{\mathring\spC^{#1}}
\newcommand{\iCs}{\iCk{\sobex+2}}
\newcommand{\iCsm}{\iCk{\sobex+1}}
\newcommand{\iCsmm}{\iCk{\sobex}}
\newcommand{\Fmap}{f}
\newcommand{\lFmap}{\dot\Fmap}
\newcommand{\cuti}{\chi_0}
\newcommand{\cuto}{\chi_\infty}

\def\defeq{:=}
\newcommand{\nperp}{\nabla^\perp}
\newcommand{\subeq}[2]{\mathord{\underbrace{\mathop{#1}}_{#2}}}
\newcommand{\supeq}[2]{\mathord{\overbrace{\mathop{#1}}^{#2}}}
\def\eqdef{=:}
\newcommand{\boi}[2]{{]#1,#2[}}
\newcommand{\loi}[2]{{]#1,#2]}}
\newcommand{\roi}[2]{{[#1,#2[}}
\newcommand{\cli}[2]{{[#1,#2]}}
\newcommand{\Lap}{\Delta}

\renewcommand{\vec}[1]{\mathbf{#1}}
\newcommand{\ve}{\vec e}
\newcommand{\vey}{\ve_y}
\newcommand{\vex}{\ve_x}
\newcommand{\rad}{r}
\newcommand{\vn}{\vec n}
\newcommand{\vs}{\vec s}
\newcommand{\xx}{{\vec x}}
\newcommand{\xxc}{\xx_0}
\newcommand{\xxi}{{\vec a}}
\newcommand{\pola}{\theta}
\newcommand{\ssnd}{c}
\newcommand{\Mach}{M}
\newcommand{\Mfac}{\overline\Jv}
\newcommand{\ps}{\pd s}
\newcommand{\pl}{\pd\lrv}
\newcommand{\pr}{\pd\rad}
\newcommand{\po}{\pd\pola}
\newcommand{\px}{\pd x}
\newcommand{\py}{\pd y}

\newcommand{\vpoti}{\vpot_\infty}
\newcommand{\Machi}{\Mach_\infty}
\newcommand{\vvi}{\vv_\infty}
\newcommand{\vi}{\vx_\infty}
\newcommand{\ssndi}{\ssnd_\infty}
\newcommand{\pivi}{\piv_\infty}
\newcommand{\densi}{\dens_\infty}
\newcommand{\ppi}{\pp_\infty}
\newcommand{\stfi}{\stf_\infty}

\renewcommand{\sobex}{s}
\renewcommand{\sobs}{\soba\sobex_{\lrv\pola}}
\newcommand{\sobh}{\soba{\sobex-\half}_\lrv}

\section{Introduction}

\begin{figure*}[h]
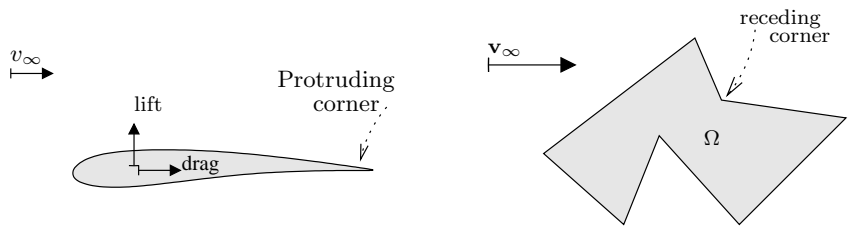
%
  \parbox{2.5in}{\input{kj.pstex_t}}\hfill\parbox{3in}{\input{polygonflow.pstex_t}}%
  \caption{Left: flow around a body that is smooth except for one protruding corner;
    right: flow around a polygon.}%
  \label{fig:polygon}%
\end{figure*}%

Consider steady planar flow around bounded solids (see fig.\ \ref{fig:polygon}). 
A corner in the solid is called \defm{protruding} if it has exterior (fluid-side) angle greater than $180^\circ$,
\defm{receding} if the angle is less than that. 
Regarding solids with three or more protruding corners, \cite{elling-polylow} shows 
in the special case of nondegenerate polygons that nonzero incompressible potential\footnote{irrotational inviscid} 
flows cannot have bounded velocity,
and that nonzero uniformly subsonic potential flows with sufficiently low Mach number do not exist.

This prior result is rather limited; in particular subsonic flows with non-small Mach number are not ruled out. 
In this paper we prove a more satisfactory result: around solids whose boundary is piecewise analytic 
with at least three protruding corners, no nonzero uniformly subsonic flows exist, for polytropic pressure law with isentropic coefficient
above $1$.  

If the velocity field of an incompressible potential flow is square-integrable near a receding corner, 
it is necessarily bounded\footnote{in fact H\"older-continuous, with limit zero in the corner} there. 
This is generally no longer true if the corner is protruding, which explains their significance \cite{elling-hyp2016}. 
Compressible potential flows do not allow unbounded velocities\footnote{except for ``unusual'' pressure laws}:
mathematically the Bernoulli equation links the density to the velocity, with density undefined above a certain \defm{limit speed}. 
Physically there is no reasonable way to extend the density definition to higher velocities, since calculations using the pressure law show that 
a material parcel of fluid cannot\footnote{for the chosen model, i.e.\ neglecting viscosity, heat conduction etc., and with $\gamma>1$ equation of state} be accelerated to 
arbitrarily high velocity by moving to regions of ever lower pressure. 

Classical study of exterior flows focused on incompressible ones, especially by conformal mappings and complex analytic methods. 
For incompressible flow unbounded velocities are merely physically unappealing, but mathematically harmless, 
so that the nonexistence problems we consider were not prominent. 
During 1930--60 advances in complex analysis and the theory of planar nonlinear elliptic PDE 
allowed progress on the compressible subsonic case (\cite{frankl-keldysh-1934,shiffman-exi-potf,bers-exi-uq-potf,finn-gilbarg-uniqueness}).  
Again nonexistence was not prominent due to focus on bodies with a single protruding corner, 
where problem parameters such as \defm{circulation} can be adjusted to render the velocity finite at that corner. 
In fact the corresponding \defm{Kutta-Joukowsky condition} yielded one of the major successes of mathematical fluid dynamics, 
a formula for lift\footnote{force perpendicular to flow direction} on aircraft wings that is in reasonable agreement with experimental data
in some regimes (small angle of attack, low velocity, etc.; see e.g.\ fig.\ 6.7.10 and surrounding text in \cite{batchelor}). 

To quote Finn and Gilbarg \cite[p.\ 58]{finn-gilbarg-uniqueness}:
\begin{quote}
  ``Unlike the case of incompressible fluids, it appears very likely that in the theory of subsonic flows the Kutta-Joukowsky condition 
  need not be imposed as an added hypothesis,
  but rather is a consequence of the subsonic character of the flow.'' 
\end{quote}
Already for two protruding corners, our prior work \cite{elling-twocorner} shows nonexistence
of uniformly subsonic flows around particular profiles including flat plates for most angles of attack. 
In absence of protruding corners the Kutta condition is void, allowing infinitely many flows,
unless other conditions are added.

We observe an amusing coincidence: potential flow appears to be an inadequate model in all cases but the one that 
happens to be arguably the most interesting and valuable for applications: the case of a single protruding corner. 
It is good practice to be skeptical about ``coincidences'', 
so we are tempted to turn the observation on its head: single-corner profiles are favored by design or evolution 
since several corners generally do not allow flows that are near-potential, 
hence tend to cause higher drag\footnote{force in flow direction}; 
whereas flow in absence of corners is poorly controlled due to non-unique circulation.
This argument should not be continued too far as it is not only heuristic and vague, 
but also partially wrong: drag is desirable for some purposes. 

Approximation of profiles by polygonal curves or other cornered bodies is a popular tool, particularly for applying conformal mapping techniques. 
Such approximations generally contain too many protruding corners to permit existence of subsonic or bounded-velocity incompressible flows,
possibly in contrast to the smoother original profile. 

Whenever irrotational subsonic inviscid flows do not exist, there are several alternatives. 
The most natural one is to consider that in reality vorticity is shed from obstacles; 
this is the main mechanism for generating \defm{drag} in the low-viscosity low-Mach regime. 
Another option is to consider transonic solutions, i.e.\ with supersonic regions at the solid. 
Such \defm{supersonic bubbles} are well-known in subsonic but nearly sonic flow at smooth outwardly curved boundary parts; 
protruding corners can be considered a limit case of those. 

In section \ref{section:equations} we review the necessary PDE and models. 
In section \ref{section:regularity} we recall more or less well-known regularity results for compressible potential flow.
In section \ref{section:local} we analyze the local structure of the \defm{body streamline};
in section \ref{section:global} we prove the main
Theorem \ref{th:threecorneranalytic}, which also provides some detailed information about incompressible flows 
and the attachment of body streamlines to protruding corners.

\section{Equations}
\label{section:equations}%

The \defm{isentropic Euler} equations are
\begin{alignat}{5} 
0 &= \pt\dens + \ndiv(\dens\vv),  \label{eq:tmass}\\
0 &= \pt(\dens\vv) + \ndiv(\dens\vv\tensor\vv) + \nabla\pp,  \label{eq:tmom}
\end{alignat} 
where $\vv$ is velocity, $\pp$ pressure. 
We focus on the \defm{polytropic} pressure law 
\begin{alignat}{5} \pp=\ppf(\dens) = \dens^\isenc \label{eq:p-polytropic} \end{alignat} 
with \defm{isentropic coefficient} $\isenc>1$, but our discussion can be extended to some other analytic $\ppf$ as well. 
Using \eqref{eq:tmass} we can transform \eqref{eq:tmom} to
\begin{alignat}{5} 0 = \Dt\vv + \nabla\piv \qquad(\Dt=\pt+\vv\dotp\nabla) \label{eq:v}\end{alignat} 
where  $\piv=\pif(\dens)$ is defined (up to an additive constant) by
\begin{alignat}{5} \pif_\dens(\dens) = \dens^{-1}\ppf_\dens(\dens) \quad. \label{eq:pp-piv}\end{alignat} 
The \defm{speed of sound} is 
\begin{alignat*}{5} \ssnd = \sqrt{\ppf_\dens(\dens)}. \end{alignat*} 

For smooth flow, if the \defm{vorticity} $\vort=\nabla\crossp\vv$ is zero everywhere at one time, then it is zero at all times. 
Such flows are called \defm{irrotational} or \defm{potential}. 
$\ncurl\vv=0$ implies
\begin{alignat}{5} \vv = \nabla\vpot \label{eq:vpot}\end{alignat} 
for a scalar \defm{velocity potential} $\vpot$ (which is locally defined and may be multivalued when extended to non-simply-connected domains). 

Henceforth we focus on stationary flow:
\begin{alignat}{5} 0 &= \ndiv(\dens\vv) \label{eq:mass} \quad, \\
 0 &= \vv\dotp\nabla\vv+\nabla\piv \quad. \notag \end{alignat} 
We substitute \eqref{eq:vpot} into the second equation to obtain\footnote{with $\vv^2=\vv\vv^T$ and $\hess$ the Hessian operator}
\begin{alignat*}{5} 0 = \nabla\vpot^T\hess\vpot + \nabla(\pif(\dens)) = \nabla\big( \half|\nabla\vpot|^2 + \pif(\dens) \big) \quad. \end{alignat*} 
This implies the \defm{Bernoulli relation}
\begin{alignat}{5} \half|\vv|^2+\pif(\dens) = \text{Bernoulli constant}  \label{eq:ber}\end{alignat} 
which is constant globally, not just along streamlines.
$\pif_\dens(\dens)=\dens^{-1}\ppf_\dens(\dens)=\dens^{-1}\ssnd^2>0$, so $\pif$ is strictly increasing, and we may solve for
\begin{alignat}{5} \dens = \ipif \big( \text{Bernoulli constant} - \half|\vv|^2 \big) \label{eq:pividens}\end{alignat} 
for some maximal interval of $|\vv|$ closed at its left endpoint $0$. 

Substituting \eqref{eq:pividens} into \eqref{eq:mass} yields \defm{compressible potential flow}, a second-order scalar differential equation for $\vpot$.
Assuming sufficient smoothness it can be expanded to\footnote{with Schur product $A:B=\trace(A^TB)$; note $A:\vec w^2=\vec w^TA\vec w$}
\begin{alignat}{5} 0 = \big(I-(\frac{\vv}{\ssnd})^2\big):\hess\vpot = \big(1-(\frac{v^x}{\ssnd})^2\big)\vpot_{xx} - 2\frac{v^x}{\ssnd}\frac{v^y}{\ssnd}\vpot_{xy} 
+ \big(1-(\frac{v^y}{\ssnd})^2\big)\vpot_{yy} \label{eq:comp-potf}\end{alignat} 
where $\ssnd$ is a function of $\dens$, hence of $\vv=\nabla\vpot$. 
The eigenvectors of the coefficient matrix $I-(\vv/\ssnd)^2$ are $\vv$ and\footnote{$\perp$ counterclockwise rotation by $\pi/2$} $\vv^\perp$, 
with eigenvalues $1-\Mach^2$ and $1$ where 
\begin{alignat*}{5} \Mach \defeq \vlen\vv/\ssnd \end{alignat*} 
is the \defm{Mach number}. 
\eqref{eq:comp-potf} is elliptic in a given point if and only if 
\begin{alignat*}{5} \Mach < 1 \quad, \end{alignat*} 
i.e.\ if and only if velocity $|\vv|$ is below the speed of sound $\ssnd$; such flows are called \defm{subsonic}. 

Instead we use the \defm{streamfunction formulation} of irrotational flow. To this end, use conservation of mass $\ndiv(\dens\vv)=0$
to obtain\footnote{with $\nabla^\perp=(-\partial_y,\partial_x)$}
\begin{alignat*}{5} \dens\vv = -\nperp\stf \end{alignat*} 
for a scalar \defm{stream function} $\stf$. 
Consider the Bernoulli relation \eqref{eq:ber} in the form
\begin{alignat}{5} \text{Bernoulli constant} = \subeq{\supeq{\half|\dens\vv|^2}{\vms}~\dens^{-2}+\pif(\dens) }{\eqdef \rhsf(\dens,\vms)} \quad \label{eq:berstf} \end{alignat} 
and apply the implicit function theorem.
At solutions $(\dens,\vms)$ of \eqref{eq:berstf} 
that are vacuum-free ($\dens>0$) and subsonic,
\begin{alignat*}{5} 
\frac{\partial\rhsf}{\partial\vms}
&=
\dens^{-2} > 0 \quad\text{and}
\\
\frac{\partial\rhsf}{\partial\dens}
&=
-\dens^{-3}|\dens\vv|^2+\pif_\dens(\dens)
=
\dens^{-1}(c^2-|\vv|^2) > 0 \quad,
\end{alignat*} 
so we obtain a solution
\begin{alignat*}{5} \frac1\dens = \hdiv(\vms) \end{alignat*} 
for a strictly increasing function $\hdiv$ defined for $\vms$ in some maximal interval $\cli{0}{\vmssonic}$ for some constant $\vmssonic\in\loi0\infty$; for $\vms=\vmssonic$ the velocity is exactly sonic.

Having solved the mass and Bernoulli equations it remains to ensure irrotationality:
\begin{alignat}{5} 0 = \ncurl \vv = \ncurl \frac{-\nperp\stf}{\dens} = -\ndiv\Big( 
\hdiv(\frac{|\nabla\stf|^2}{2}) \nabla\stf 
\Big) \label{eq:stf-divform} \end{alignat} 
Assuming sufficient additional regularity, differentiation yields after some calculation that
\begin{alignat}{5} 
0
&=
\big(1-(\frac{\vv}{\ssnd})^2\big):\hess\stf = \big(1-(\frac{v^x}{\ssnd})^2\big)\stf_{xx} - 2\frac{v^xv^y}{\ssnd^2}\stf_{xy} + \big(1-(\frac{v^y}{\ssnd})^2\big)\stf_{yy} 
\label{eq:stf-2d}
\end{alignat} 
which has the same coefficient matrix as \eqref{eq:comp-potf}; again it is elliptic if and only if the flow is subsonic. 

The incompressible limit of \eqref{eq:stf-2d} is obtained by (for example) considering sequences of solutions with velocities 
approaching $0$, hence sound speed and density converging to positive constants. This yields
\begin{alignat}{5} 0 &= \Lap\stf. \label{eq:harmonic}\end{alignat}

At solid boundaries we use the standard \defm{slip condition}
\begin{alignat}{5} 0 = \vn \dotp \vv\quad, \label{eq:slip}\end{alignat} 
where $\vn$ is a normal to the solid. 
Using $\dens\vv=-\nperp\stf$ we obtain 
\begin{alignat}{5} \stf = \const \label{eq:stf-const} \end{alignat} 
on each connected component of the slip boundary, in our case only one, so we may take $\stf=0$ by adding a constant to $\stf$
which does not affect \eqref{eq:stf-divform}.

\section{Regularity}
\label{section:regularity}

\begin{figure*}
  \centerline{\input{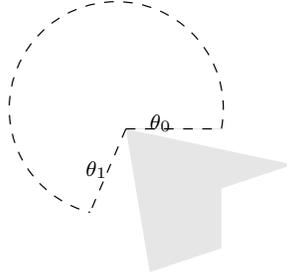}}
  \caption{A protruding corner is the center of a pacman (circular sector with fluid-side angle $>\pi$)}
  \label{fig:pacman}
\end{figure*}

We first state precise assumptions about our solid, in particular defining ``piecewise analytic'' clearly:
\begin{definition}
  \label{def:setting}%
  The solid body $\Body\subset\R^2$ is boundec closed connected nonempty;
  $\Dom\defeq\R^2\setdiff\Body$ is the set of \defm{fluid points}. 
  $\Bdry$ is ``piecewise analytic'' in the following 
  sense:
  it is a union of finitely many simple curves that are analytic\footnote{``regular'' analytic curves, not varieties} including endpoints, 
  and pairwise disjoint except perhaps at their endpoints.
  We call the endpoints \defm{corners}, the other boundary points are called \defm{smooth}. 
  A corner is \defm{protruding} if it is the center of a \defm{pacman} (fig.\ \ref{fig:pacman}), i.e.\ an open circular sector with angle greater than $180^\circ$
  and contained in the fluid domain $\Dom$. 
  
  A \defm{compressible uniformly subsonic potential flow} is represented by a distribution $\stf$ on $\Dom$ with distributional derivatives in $\Linf(\Dom)$ 
  satisfying \\
  1. $\esssup_\Dom\Mach<1$ (uniformly subsonic), \\
  2. the slip condition $\stf(\xx)\conv0$ as\footnote{by our assumptions on $\Dom$ any boundary point $\xx_0$ is the endpoint of a line segment whose interior is through the fluid $\Dom$} $\xx\conv\xx_0$ for any $\xx_0\in\Slipb$, an\\
  3. the partial differential equation \eqref{eq:stf-divform}; an \defm{incompressible potential flow} satisfies \eqref{eq:harmonic} instead. 
\end{definition}
For the remainder of the paper we consider only potential flows that are \defm{nonzero}, meaning $\vv\neq 0$. 

\begin{remark}
  The definition allows parts of the body to have ``zero thickness'' (see fig.\ \ref{fig:tripod}), 
  with one of the analytic curve making up the boundary having fluid on both sides. 
\end{remark}

Since the equation is uniformly elliptic, with analytic coefficients,
we may apply Morrey estimates (for example)
as well as bootstrapping Schauder estimates to obtain $\Cinf(\cDom\setdiff\Corners)$ regularity
 (see also \cite[Proposition 2]{elling-polylow} and \cite{elling-protrudingangle}). 

Moreover $\stf$ is also analytic in $\Dom$, 
by classical results\footnote{see e.g.\ \cite[Chapitre IV, first ``Th\'eor\`eme'']{bernstein-1904} and the references in the introduction;
 see also \cite{friedman-analyticity,morrey-analyticity-i}},
since the pressure function $\ppf$ is analytic \eqref{eq:p-polytropic} and hence so are the coefficient functions. 
Using \cite{morrey-analyticity-ii} it is possible to extend the result to analytic parts of the boundary, but it is just as easy to do it ``by hand'':

Given a boundary point so that the boundary is an analytic curve in a neighbourhood of the point. 
We may choose coordinates so that the curve is parametrized as $y=f(x)$ for an analytic function $f$, with $y>f(x)$ a fluid region. 
An analytic change of coordinates $w=y-f(x)$ maps the boundary to $\set{w=0}$, yielding a new uniformly elliptic PDE whose coefficient functions are still analytic (and now dependent on $\stf(x,w)$ and $x,w$ as well). 
We extend $\stf$ from $w>0$ by odd reflection: for $w<0$, set $\stf(x,w)=-\stf(x,-w)$. The resulting function has no jump, since $\stf=0$ on the slip boundary, and its normal ($\pd w$) derivative has no jump either, by odd reflection. 
Hence it is locally $C^{1,1}$ (note that we already obtained $\Cinf$ on the $w\geq 0$ part of the neighbourhood),
so the coefficients, which are still functions of $x,w,\stf,\stf_x,\stf_w$, but not second derivatives, are $C^{0,1}$, in particular $C^{0,\eps}$ for some $\eps>0$, as needed to bootstrap Schauder estimates. 
We obtain that the extended $\stf$ is in fact analytic, which it remains after passing back to $x,y$ coordinates. 
Hence:
\begin{proposition}
  \label{prop:base-regularity}%
  $\stf$ is analytic in $\Dom$ as well as at each fluid side of analytic boundary segments. 
\end{proposition}
Analyticity also implies that $\stf$ cannot be locally zero anywhere, because then it would be globally zero, but we excluded zero flows. 

Finally, it is well-known (see \cite{bers-exi-uq-potf,finn-gilbarg-uniqueness} and \cite[Proposition 1]{elling-twocorner}) 
that $\nabla\stf$ is ``H\"older-continuous at infinity''; more precisely we may rotate coordinates, as is standard in the literature, so that 
the limit of the velocity $\vvi$ at infinity is $(\vxi,0)$ with constant $\vxi\geq 0$, and then:
\begin{proposition}
  \label{prop:stf-infinity}%
  \begin{alignat}{5} \stf(\xx)= \densi\big(\vxi y - \frac{\Gam}{2\pi} \pglau \log\sqrt{x^2+(\pglau y)^2}\big) + \const + o(1)  \quad\text{as $|\xx|\conv\infty$} 
  \label{eq:stf-infinity}\end{alignat}
  where $\pglau=\sqrt{1-\Machi^2}$ is the \defm{Prandtl-Glauert factor}, whereas $\Gam$ is the \defm{circulation};
  the $o(1)$ term has $o(1)$ derivative. 
\end{proposition}

\section{Local structure of the body streamline}
\label{section:local}

In contrast to our earlier work \cite{elling-polylow}, we choose to analyze not the sets $\Ps,\Ns$ of positive and negative stream function values, 
but rather the \defm{body\footnote{``body'' since $\stf=0$ at the solid boundary; by continuity of $\stf$ every other streamline is separated from the body} streamline} 
\begin{alignat*}{5} \Zs \defeq \selset{\xx\in\Dom}{\stf(\xx)=0}. \end{alignat*} 
(Note that as defined $\Zs$ includes only fluid points, not solid boundary points.)
Streamlines are also commonly used to construct informal arguments in applied works on fluid flow
(and some of our following arguments may have appeared implicitly or in rather different form  in some of the more applied or historical literature, especially for incompressible flow). 
But here, as so often, the most elegant and powerful ways of \emph{informal} physical reasoning turn out to be rather arduous once the necessary details lacking for a \emph{rigorous} mathematical proof are filled in; 
after all, level sets of analytic functions are sufficiently complicated to motivate large parts of complex analysis, algebra etc.

Nevertheless in our particular settings $\Zs$ turns out to be manageable.
The local structure of real-analytic varieties is well-understood \cite{lojasiewicz,bierstone-milman}, 
and we are in two dimensions where Puiseux series expansion yields a shorter approach.
Since our $\stf$ is not merely real-analytic but satisfies an explicit elliptic PDE as well, we can simplify even that process considerably. 

We first consider the structure of $\Zs$ near infinity, using the known asymptotics we described earlier. 
\begin{proposition}
  \label{prop:curves-at-inf}%
  If $\vvi\neq 0$, then 
  in some neighbourhood $\selset{(x,y)}{|x|>R\ \text{or}\ |y|>R}$ of infinity, 
  $\Zs$ consists of two analytic curves, one parametrized by $x\in\boi{-\infty}{-R}$ and one by $x\in\boi{R}{\infty}$. 
\end{proposition}
\begin{proof}
  By Proposition \ref{prop:stf-infinity}, 
  \begin{alignat*}{5} \nabla\stf=(0,\densi\vxi)+o(1) \quad\text{as $|(x,y)|\conv\infty$,} \end{alignat*} 
  with $\densi>0$ and $\vxi>0$, 
  so $\stf_y(x,y)\geq\half\densi\vxi>0$ for $|x|>R$ or $|y|>R$ with $R$ sufficiently large. 
  Therefore $\stf\neq 0$ for $|y|>R$ and $|x|\leq R$ if we increase $R$ further as needed, 
  and $\stf(x,y)=0$ for precisely one $y=\hat y(x)$ for each $x$ with $|x|>R$, 
  By the implicit function theorem for analytic functions, using $\stf_y>0$, we obtain that $x\mapsto\hat y(x)$ is real-analytic. 
\end{proof}
Henceforth we call the $x\conv-\infty$ part of $\Zs$ ``negative infinity'' and $x\conv+\infty$ ``positive infinity''.

Now we consider the structure near finite fluid points.
\begin{proposition}
  \label{prop:strong-maxp}%
  $\stf$ cannot attain local extrema in fluid points unless the flow is zero.
\end{proposition}
\begin{proof}
  This is the classical strong maximum principle \cite[Theorem 3.5]{gilbarg-trudinger}, applied to our interior PDE \eqref{eq:stf-2d}:
  if $\stf$ attains a local extremum, then it is constant, but we assumed presence of a nonempty body $\Body$, and $\stf=0$ 
  at its nonempty boundary $\Slipb$ means $\stf$ is zero everywhere.
\end{proof}

\begin{proposition}
  \label{prop:vertex}%
  For a nonzero flow, 
  near each fluid point 
  $\Zs$ is a union of $2m$ ($m\geq 1$ integer) analytic curves that are pairwise disjoint except for ending in that point, where their tangents enclose equal angles $\pi/m$
  (the antipodes form a single analytic curve passing through the point). 
  Near each \emph{smooth} boundary point, $\Zs$ is a union of $m-1$ analytic curves pairwise disjoint except for their common boundary endpoint; 
  the curve and boundary tangents enclose equal angles $\pi/m$. 
  If $m\geq 2$, then we call the point a \defm{vertex}. 
\end{proposition}
\begin{proof}[Proof of Proposition \ref{prop:vertex}]
  Let $\stf=0$ in some fluid or smooth boundary point (in the latter case we focus on one fluid side of the boundary).
  Let the point be $0$, by a simple translation of coordinates. 
  Let $m$ be minimal so that $D^m\stf(0)\neq 0$. 
  (Such an $m$ exists because we showed $\stf$ is analytic and cannot be locally zero.)
  
  By the strong maximum principle (Proposition \ref{prop:strong-maxp}) 
  $\stf$ and hence the dominant degree $m$ homogeneous part of its Taylor polynomial cannot be single-signed in a neighbourhood of $0$.
  Thus by homogeneity the degree $m$ part must be zero on some line.
  We may rotate coordinates so that the line coincides with the $x$ axis. 

  Consider the case $m\geq 2$. 
  Taking $\px^j\py^{m-2-j}$ ($0\leq j\leq m-2$) of the equation
  \begin{alignat*}{5} 0 = (I-\ssnd^{-2}\vv^2):\nabla^2\stf \end{alignat*} 
  yields, using $\nabla^j\stf(0)=0$ for $j=1,...,m-1$ and hence $\vv(0)=0$, that
  \begin{alignat*}{5} 0 = \px^{j+2}\py^{m-2-j}\stf + \px^j\py^{m-j}\stf \quad\text{in $0$.} \end{alignat*} 
  Hence, because $\px^m\stf(0)=0$ by our rotation above, 
  \begin{alignat*}{5} \px^j\py^{m-j}\stf(0)=0 \quad\text{for \emph{even} $j$,} \end{alignat*}
  and similarly 
  \begin{alignat*}{5} \px^j\py^{m-j}\stf(0) = (-1)^{(j-1)/2} a \quad\text{for \emph{odd} $j$,} \end{alignat*}
  where $a$ must be nonzero.
  Combined we obtain a Taylor expansion (with $z=x+iy$)
  \begin{alignat*}{5} \stf
  = a_2 \Im ( z^m ) + O(|z|^{m+1}) \end{alignat*}
  for some nonzero $a_2$. For $m=1$ the same expansion holds trivially after rotation. 
  The zeros of the leading term $a_2\Im(z^m)$ are obviously 
  \begin{alignat*}{5} z_k=t\exp\frac{k\pi i}{m} \quad\text{for $k=1,...,m$ and $t\in\R$,} \end{alignat*}
  lines at equal angles $\pi/m$. 

  We rotate coordinates again slightly so that none of these lines has vertical tangent. Let $s_1<...<s_m$ be the slopes, then 
  \begin{alignat*}{5} \stf = a_3\cdot(y-s_1x)\cdot...\cdot(y-s_mx) + O(|\xx|^{m+1}) , \end{alignat*} 
  for some nonzero $a_3$. For $x\neq 0$ a division by $x^m$ yields with $s=y/x$ that 
  \begin{alignat*}{5} \stf x^{-m} = \subeq{ a_3\cdot(s-s_1)\cdot...\cdot(s-s_m) + xR(s) }{\eqdef h(x,s)} \csep \text{$R(s)$ real-analytic.} \end{alignat*} 
  $\stf=0$ is then equivalent to $0=h(x,s)$ with $h$ analytic and satisfying $\partial h/\partial s(0,s_k)\neq0$ for every $k=1,...,m$.
  Hence the implicit function theorem (e.g.\ \cite[Theorem 2.3.5]{krantz-parks}) 
  shows that there are real-analytic functions $\hat s_k$, defined for $x$ near $0$ with $\hat s_k(0)=s_k$, 
  so that $\stf(x,\hat s_k(x))=0$. Each $\hat s_k$ parametrizes one pair of the desired curves (in the boundary case we retain only the $m-1$ curves 
  on the fluid side we chose to consider). 
\end{proof}

\newcommand{\axx}{a^{xx}}
\newcommand{\axy}{a^{xy}}
\newcommand{\ayy}{a^{yy}}
\newcommand{\maxr}{\overline r}
\newcommand{\subsocoeffeps}{\delta}
\newcommand{\subsoexpeps}{\eps}
\newcommand{\polamid}{\pola_c}

We have obtained rather detailed information about the zero streamline away from the corners. 
Structure near the corners could be clarified after obtaining regularity results there. 
However, it is actually possible to bypass corner regularity analysis altogether. 
To this end we note an important observation about protruding corners (see fig.\ \ref{fig:pacman}):
\begin{proposition}
  \label{prop:singlesign}%
  For a compressible uniformly subsonic nonzero potential flow, 
  we cannot have $\stf\geq 0$ (or $\stf\leq 0$) in a pacman.
  This is also true for incompressible nonzero potential flows if their velocity is bounded near the corner.
\end{proposition}
In particular every corner must be in the closure of $\Zs$. 
\begin{proof}
  \begin{enumerate}
  \item
    Our flow corresponds to a stream function $\stf$ solving $L\stf=0$ with operator
    \begin{alignat*}{7} L = -A(\xx):\nabla^2 = -\axx(\xx)\pd x^2 - 2\axy(\xx)\pd x\pd y - \ayy(\xx)\pd y^2 \end{alignat*}
    uniformly elliptic on the pacman $U=\set{\polalo<\pola<\polahi,\ 0<\rad<\maxr}$ 
    (with $\maxr>0$, $\polahi-\polalo\leq 2\pi$) in polar coordinates $(\rad,\pola)$ centered in the protruding corner. 
    Protruding means $\polahi-\polalo$ is greater than $\pi$.

    We claim existence of a subsolution $\subso$ with
    \begin{alignat}{7} \subso &\geq \iota\rad^{1-\eps} \quad\text{on some ray $\set{\pola=\polamid,\ 0<\rad<\maxr}$}  \label{eq:infgrad}\end{alignat} 
    for some constants $\eps\in\boi01$, $\iota>0$ and $\polamid\in\boi{\polalo}{\polahi}$.
    More precisely, $\subso$ is $\Ctwo$ in the pacman $U$ and continuous on its closure, with
    \begin{alignat}{7} 
      L\subso &\leq 0 \quad\text{in the pacman, and} \label{eq:subsolprop} \\
      \subso &\leq 0 \quad\text{on the radii $\set{\pola=\pola_q,\ 0\leq\rad\leq\maxr}$ for $q=0,1$.} \label{eq:zeroonradii}
    \end{alignat}
    We obtain $\subso$ by the ansatz
    \begin{alignat*}{7} \subso(\rad,\pola) &= \rad^{1-\eps} f(\pola). \end{alignat*} 
    At $\pola=0$, $L\subso\leq 0$ is 
    \begin{alignat*}{7} 0 
      &\leq 
      \big( \axx\pr^2 + 2\axy\rad^{-1}(\pr-\rad^{-1})\po + \ayy(\rad^{-2}\po^2+\rad^{-1}\pr) \big) \subso
      \\&=
      \rad^{-1-\eps} \Big( \ayy ( f + f_{\pola\pola} ) -\eps \big( \axx (1-\eps) f  + 2\axy f_\pola + \ayy f \big) \Big)
    \end{alignat*} 
    and same at other $\pola$ if the $A$ coefficients are rotated accordingly. 
    To satisfy the inequality \eqref{eq:subsolprop} it is sufficient to solve $f+f_{\pola\pola}=1$ 
    and then take $\eps>0$ sufficiently small, 
    using $|\axx|,|\axy|\leq C\ayy$ 
    for some constant $C<\infty$ independent of $\pola$, by uniform ellipticity. 
    The solutions are
    \begin{alignat*}{7} f=1+a\cos(\pola-\polamid) \end{alignat*}
    We pick $a\geq 1$ so that on each side of the maximum $\pola=\polamid$ we have zeros in 
    \begin{alignat*}{5} \pola = \polamid \pm \tilde\pola \csep \tilde\pola = \arccos\frac{-1}{a}=\pi-\arccos\frac1a. \end{alignat*}
    The distance $\tilde\pola$ ranges from $\pi$ to arbitrarily close to but larger than $\pi/2$ as $a$ ranges from $1$ to $\infty$.
    Hence we may take $\polamid=\half(\polalo+\polahi)$ and $a$ so that $\polamid-\tilde\pola=\polalo$ while $\polamid+\tilde\pola=\polahi$. 
    Then $f=0$ in $\pola=\polalo$ and in $\pola=\polahi$ so that \eqref{eq:zeroonradii} holds, and $f>0$ for $\pola\in\boi{\polalo}{\polahi}$,
    in particular in $\pola=\polamid$, so that \eqref{eq:infgrad} holds. 
  \item
    Now we apply a comparison principle argument. 
    Assume, contrary to our claim, that $\stf\geq 0$ on the closure of the pacman.
    Then by the strong maximum principle $\stf>0$ in its interior.
    We may shrink the pacman slightly, keeping its interior angle greater than $\pi$, so that $\stf>0$ on the closure of the pacman except in the center. 

    $\stf>0$ on the compact arc $\set{\polalo\leq\pola\leq\polahi,\ \rad=\maxr}$
    where $\subso$ and $\stf$ are continuous, hence $\stf$ is uniformly positive and $\subso$ uniformly bounded there,
    so by taking $\tilde\iota>0$ sufficiently small we have $\tilde\iota\subso\leq\stf$ on the arc.
    That holds on the entire boundary of the pacman because on the two radii we have $\tilde\iota\subso\leq 0$ by construction as well as $\stf\geq0$ (in fact $>0$ except in the center endpoint). 
    Moreover in the interior $L(\tilde\iota\subso)\leq 0$ and $L\stf\geq 0$ (in fact $=0$), so the comparison principle implies $\stf\geq\tilde\iota\subso$ on the pacman closure. 

    In particular 
    $\stf\geq \tilde\iota\iota\rad^{1-\eps}$ on the $\pola=\polamid$ ray from the body corner (pacman center) where $\stf=0$, 
    which contradicts boundedness of $\nabla\stf$. 
    The contradiction shows our assumption that $\stf\geq 0$ was wrong. The case $\stf\leq 0$ is analogous.
  \end{enumerate}
\end{proof}

\section{Global structure of the body streamline}
\label{section:global}

Combined with the known asymptotics at infinity we immediately obtain the following consequence:
\begin{proposition}
  \label{prop:vvinonzero}
  If there is at least one protruding corner, then nonzero flows must have $\vvi\neq 0$. 
\end{proposition}
\begin{proof}
  Assume $\vvi=0$, then 
  \begin{alignat*}{5} \stf \topref{eq:stf-infinity}{=} - \frac{\densi\Gam\pglau}{2\pi}\log\sqrt{x^2+(\pglau y)^2} + \const + o(1) \end{alignat*} 
  The $\log$ is positive near infinity where it dominates the $\const$ and $o(1)$ terms, 
  so if its coefficient is nonzero, then $\stf$ is single-signed near infinity. 
  Since $\stf=0$ at the body, the strong maximum principle implies it is single-signed everywhere.
  But that contradicts Proposition \ref{prop:singlesign} since we assumed presence of protruding corners. Hence the coefficient is zero and the $\log$ term disappears.

  The same argument shows that the now-dominant ``$\const$'' term is also zero. The last term vanishes at infinity,
  so the strong maximum principle shows that $\stf=0$, so that the flow is zero. 
\end{proof}

Another key consequence of the strong maximum principle is the following:
\begin{proposition}
  \label{prop:no-jordan}%
  For a nonzero flow, 
  \begin{enumerate}
  \item $\Zs$ does not contain simple closed curves. 
  \item A simple curve in $\Zs$ cannot have both ends converging to the body.
  \end{enumerate}
\end{proposition}
\begin{figure*}
  \input{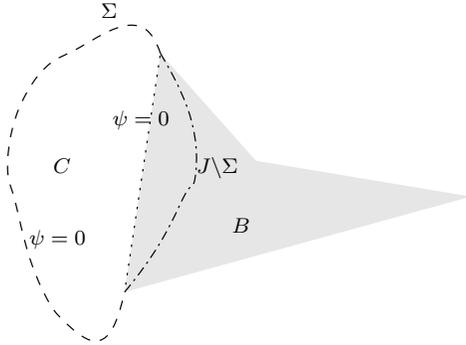}
  \caption{An $\stf=0$ curve connecting two body points (here corners) contradicts the strong maximum principle}
  \label{fig:bodycurve}
\end{figure*}
\begin{proof}
  If a simple curve $\Sigma$ in $\Zs$ converges to body points at both ends (see fig. \ref{fig:bodycurve}), then we may extend it to a simple closed curve $J$
  through the (path-connected) body $\Body$. 
  The case where $\Zs$ contains a simple closed curve $\Sigma=J$ (which may or may not enclose $\Body$) is a special case. 
  
  By the Jordan curve theorem
  $\R^2\setdiff J$ has two open connected components,
  one unbounded, the other bounded, each having boundary $J$. Call the bounded component $C$;
  it must contain fluid points since the $\Sigma$ part of its boundary $J$ does, so $C\isect\Dom$ is nonempty.
  Every point on $\bdry(C\isect\Dom)$ is either a point of the solid boundary $\bdry\Dom=\bdry\Body$, where $\stf=0$ by slip condition, 
  or a point on $\bdry C=J$ away from $\Body$ and hence on the original curve $\Sigma\subset\Zs$ where $\stf=0$ by definition of $\Zs$.
  Hence $\stf$ must attain its extremum over $\overline{C\isect\Dom}$ in the interior $C\isect\Dom$ --- contradiction to the strong maximum principle. 
\end{proof}
That connected components $\Zs$ cannot contain cycles is a strong constraint.
In essence it reduces $\Zs$ to a ``tree''\footnote{standard definitions of ``embedded graph'' do not quite apply without proving more regularity at the corners}.
In our particular case the tree reduces further to a single curve, but the techniques are also useful in other applications where the trees may be more complex.

\begin{proposition}
  \label{prop:single-curve}%
  If $\vvi\neq 0$, then $\Zs$ is either
  \begin{enumerate}
  \item[a.]
    a single analytic curve from negative to positive infinity that does not meet the body, or
  \item[b.]
    a union of two disjoint analytic curves, one from each infinity converging at the other end to a unique body point (possibly the same). 
  \end{enumerate}
\end{proposition}
\begin{proof}
  \begin{enumerate}
  \item
    By Proposition \ref{prop:vertex} a curve in $\Zs$ starting in some vertex cannot just stop somewhere, 
    so we can continue it either to another vertex or to infinity or to the solid body. More precisely: 
  consider a unit-speed parametrization $s\mapsto z(s)$ of the local curve starting from that vertex, with arc length $s=0$ in the vertex. 
  Consider the $s>0$ side analytic extension; for the extended $s$ interval $\roi{0}{s_*}$ pick $s_*$ maximal so that 
  $z(s)$ is a non-vertex fluid point for all $s\in\roi{0}{s_*}$. 
  If there is a sequence $(s_n)\upconv s_*$ so that $z(s_n)$ converges to a point in $\Dom$, then 
  by Proposition \ref{prop:vertex} $z(s)$ itself must converge as $s\upconv s_*$, and the limit is in $\Zs$;
  it must be a vertex because in non-vertex points $\Zs$ is locally a single analytic curve, 
  so we could have continued the extension within $\Zs\setdiff\set{\text{vertices}}$,
  contradicting maximality.
  If there is no sequence $(s_n)$ as above, then $z(s)$ must converge either to infinity or to a body point as $s\upconv s_*$. 

  We call such a maximal curve an edge. 
\item
  Consider any connected component $C$ of $\Zs$. 
  Assume $C$ contains a fluid vertex. 
  By Proposition \ref{prop:vertex} there are $m\geq 4$ edges ending in that vertex. 
  Consider one of them. 

  If it converges to another fluid vertex at the other end, Proposition \ref{prop:vertex} permits continuing along another edge, 
  for definiteness say the first one in clockwise direction from the arrival edge.
  We repeat this as many times as possible, finitely many times if the last path edge converges to a body point
  or one of the infinities, infinitely many times otherwise. 
  The resulting path must be simple because Proposition \ref{prop:no-jordan} rules out cycles. 
\item
  Consider a path not ending on the \emph{smooth} part of the body boundary.
  For $\epsilon>0$ consider neighbourhoods $\set{x<-1/\epsilon}$ and $\set{x>1/\epsilon}$ of the two infinities;
  for corner neighbourhoods we choose open balls of radius $\epsilon$ centered in the corner. 
  Definition \ref{def:setting} permits only finitely many corners, so for sufficiently small $\epsilon>0$ 
  the chosen neighbourhoods are pairwise disjoint. 

  Let $\Ue$ be the union of the neighbourhoods; it is open. Then $\Ke\defeq\overline\Zs\setdiff\Ue$ is closed, 
  and by Proposition \ref{prop:curves-at-inf} it is also bounded, hence compact. 
  By Proposition \ref{prop:vertex} every point of $\Ke$ is the center of a ball in which $\Zs$ consists of $2m$ simple analytic curves from that point. 
  A finite number of these balls cover $\Ke$; let $\We$ be their union. 
  
  The edge maximality argument we gave also shows that an edge that meets $\We$, hence meets one of the finitely many balls constituting $\We$, must visit the center of that ball. 
  $\We$ contains only finitely many centers, and the path must visit one of them each time it visits $\We$, so since the path is simple
  it must eventually leave $\We$ and not return. 
  Since the 
  neighbourhoods are pairwise disjoint, 
  the path must eventually enter one of them without leaving again. 

  By taking $\epsilon\dnconv 0$, with neighbourhoods monotonically decreasing, we see that the path must converge to that corner or infinity. 
  \item
  We obtain at least four such paths, meeting only in the original vertex, 
  so by Proposition \ref{prop:curves-at-inf} at most one can converge to each infinity. 
  That leaves at least two others converging to body points at their other ends, 
  and since they have one end in common their union is a simple path from the body back to the body, in contradiction to Proposition \ref{prop:no-jordan}. 

  Hence our assumption was wrong: the connected component $C$ of $\Zs$ does not have vertices, it is a simple analytic curve. 
  If it is bounded, then again both ends converge to the body, causing a contradiction. Hence there are only unbounded components. 

  Since each component curve has an end converging to one infinity, there cannot be more than two. If there are two, then each converges to a body point at the other end 
  (not necessarily to distinct ones).
  If there is one component, it must pass from negative to positive infinity without meeting the body. 
  \end{enumerate}
\end{proof}

\begin{proposition}
  \label{prop:protruding-curve}%
  For compressible uniformly subsonic nonzero potential flows, 
  every protruding corner of the body has at least one $\Zs$-curve converging to it.
  This is also true for nonzero incompressible potential flows if their velocity is bounded in the fluid pacman.
\end{proposition}
\begin{proof}
  By Proposition \ref{prop:singlesign} $|\nabla\stf|$ would be unbounded at the corner
  unless the corner was in the closure of one of the connected components of $\Zs$. 
  If that component is not separated from the body, then by Proposition \ref{prop:single-curve} it must converge to a unique body point, necessarily the corner. 
\end{proof}

Combining Proposition \ref{prop:protruding-curve} and Proposition \ref{prop:single-curve}, we immediately obtain
\begin{theorem}
  \label{th:threecorneranalytic}
  In the situation of Definition \ref{def:setting}: 
  there are no nonzero compressible uniformly subsonic flows, nor nonzero bounded-velocity incompressible flows, around bodies with three or more protruding corners (fig.\ \ref{fig:tripod}).
\end{theorem}

\begin{remark}
  We also obtained some information about the other cases
  (the example figures use Karman-Trefftz profiles with $\nu$ the corner exponent, $\mu$ the $\zeta$-plane circle center and $\alpha$ the velocity angle):
  \begin{enumerate}
  \item
    For two protruding corners, 
    the body streamline has two components, one from negative infinity to one protruding corner and one from the other protruding corner to positive infinity (fig.\ \ref{fig:twocorner}). 
    \cite{elling-twocorner} discusses that such flows generically do not exist, but do exist that in special cases.
  \item
    For a single protruding corner,
    the body streamline consists of two analytic curves, one from negative and one from positive infinity, one ending in the corner and the other in a unique boundary point (possibly also the corner;
    fig.\ \ref{fig:onecorner} center vs.\ left, right). 
    Flows around one-corner bodies generally exist, under various reasonable assumptions, as is already known from classical work (see \cite{finn-gilbarg-uniqueness} and references therein).
  \item
    If there are no protruding corners, then Proposition \ref{prop:vvinonzero} does not apply, so we may have nonzero flows with $\vvi=0$ (if circulation $\Gam$ is nonzero).
    If $\vvi=0$, the body streamline is the empty set (e.g.\ consider the trivial incompressible $\vv=(-y,x)/|\xx|^2$ around a unit disk body). 
    If $\vvi\neq0$, then the body streamline can be 
    \begin{enumerate}
    \item a single analytic curve from negative to positive infinity, not meeting the body (fig.\ \ref{fig:zerocorner} right), or 
    \item two analytic curves, each meeting the body in a unique point (possibly the same; fig.\ \ref{fig:zerocorner} center vs.\ left)
    \end{enumerate}
  \end{enumerate}
\end{remark}

\begin{figure*}[h]
  \noindent%
  \parbox{.45\linewidth}{%
    \input{twocorners.pstex}%
    \caption{Incompressible flow around a two-corner profile:
      $\mu=0$, $\nu\pi=270^\circ$, 
      $\alpha=0^\circ$, $\Gam$ Kutta-Joukowsky value.}
    \label{fig:twocorner}%
  }\hfil\parbox{.45\linewidth}{%
    \input{tripod.pstex_t}
    \caption{$\geq 3$ protruding corners: 
      at least one  is not connected to infinity by a streamline.}
    \label{fig:tripod}%
  }
\end{figure*}
\begin{figure*}[h]
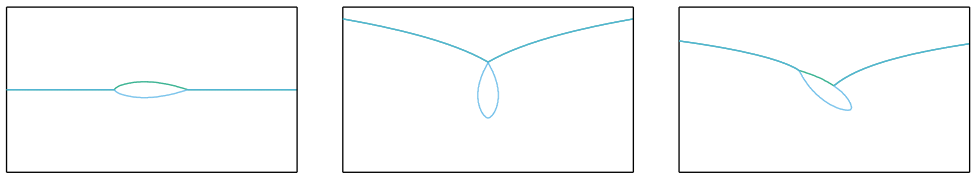

  \hfil%
  \input{onecornerright.pstex}%
  \input{onecornersame.pstex}%
  \input{onecornerleft.pstex}%
  \hfil%
  \caption{Incompressible flows around a one-corner profile:
    $\mu=-0.1$, $\nu\pi=315^\circ$, 
    $\alpha\in\{0^\circ,90^\circ,135^\circ\}$, $\Gam$ Kutta-Joukowsky value.
  }
  \label{fig:onecorner}%
\end{figure*}
\begin{figure*}[h]
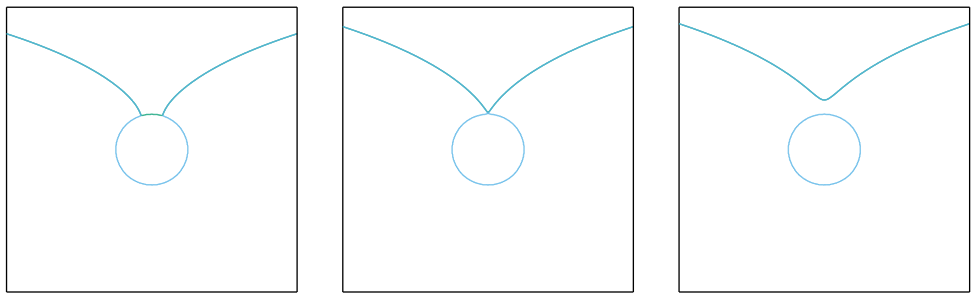

  \hfil%
  \input{zerocorners2.pstex}%
  \input{zerocorners1.pstex}%
  \input{zerocorners0.pstex}%
  \hfil%
  \caption{Incompressible flow around a unit circle for three values of the circulation:
    $\alpha=0$, 
    $\mu=0$, $\Gam\in\{12,12.57,12.8\}$. 
  }
  \label{fig:zerocorner}%
\end{figure*}
\clearpage
\begin{remark}
  Proposition \ref{prop:vertex} shows that a body streamline converging to a smooth body point meets the boundary perpendicularly (fig.\ \ref{fig:zerocorner} left), unless the other body streamline also converges to the same point,
  in which case both form an angle $60^\circ$ to the boundary and to each other (fig.\ \ref{fig:zerocorner} center).
\end{remark}

\begin{figure*}
  \input{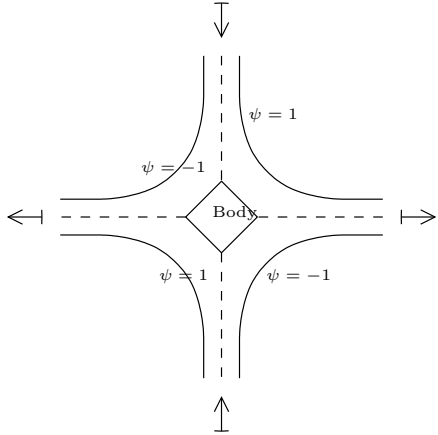}
  \caption{Flow around a four-protruding-corner body through four ``channels'' at infinity; solid lines indicate slip-condition walls, dashed lines the body streamline.}
  \label{fig:channel}
\end{figure*}

\begin{remark}
  Our nonexistence results for three or more protruding corners are somewhat ``topological'' in nature: 
  they use crucially that the flow at infinity generates only two ``body streamlines'', two being insufficient to accomodate more than two protruding corners. 

  If we alter the problem, e.g.\ by restricting infinity to four channels with walls that are straight and parallel (see fig.\ \ref{fig:channel}),
  then bodies with more protruding corners are certainly possible. 
  E.g.\ for incompressible flow we focus on the upper right quadrant (above and right of the dashed lines); we solve $\Lap\stf=0$ with $\stf=1$ on the curved boundary,
  $\stf=0$ on the straight diagonal side of the diamond (body) and on the upper and right dashed lines;
  corner angles $<\pi$ yield bounded $\nabla\stf$. 
  Then perform an odd reflection across the positive vertical axis and another odd reflection across the entire horizontal axis to complete the construction. 
  
  Similarly, if there are $n$ disjoint compact bodies rather than just one,
  then we have $n$ circulation-type free parameters to adjust to keep the velocity bounded at more protruding corners. 

  However, in many conceivable applications, for example numerical approximation of curved bodies by finely subdivided polygons,
  the number of protruding corners will easily exceed the number of fortunate symmetries, free parameters or ``infinities''. 
\end{remark}

\begin{remark}
  Our proof technique leads to a rather transparent proof;
  on the other hand it is currently limited to analytic equations of state and piecewise analytic boundaries,
  which seems to cover all approximations commonly used in numerics. 
  The detailed analysis of Bers \cite{bers-exi-uq-potf} may offer sufficient tools for obtaining non-existence results in cases of $\Cinf$ or lower regularity. 
\end{remark}

\section*{Acknowledgements}

This material is based upon work partially supported by Taiwan MOST grant 105-2115-M-001-007-MY3.

\end{document}